 \PassOptionsToPackage{square,numbers,sort&compress}{natbib}
 \documentclass[final,1p,times,twocolumn]{elsarticle}

\usepackage{graphics}
\usepackage{graphicx}
\usepackage{epsfig}

\usepackage{amsthm,amssymb,amsmath,lineno,hyperref}
\modulolinenumbers[5]
\theoremstyle{definition}
\newtheorem{definition}{Definition}[section]
\theoremstyle{theorem}
\newtheorem{theorem}[definition]{Theorem}

\newtheorem{proposition}[definition]{Proposition}
\newtheorem{corollary}[definition]{Corollary}

\newtheorem{assumption}[definition]{Assumption}
\theoremstyle{definition}

\usepackage{subfig}
\usepackage{float}

\begin{document}

\begin{frontmatter}


\title{  Malliavin differentiability and regularity of densities in semi-linear stochastic delay equations driven by  weighted fractional Brownian motion}
%%\author[label1,label2,label3]{F.Mahmoudi, R. Ezzati, M. Khodabin}
\author{ Mahdieh Tahmasebi}
\address{Department of applied Mathematics, Tarbiat Modares University, P.O. Box 14115-134, Tehran, Iran.} %%\address[label2]{R. Ezzati}, \address[label3]{M. Khodabin}
%%\centerline{ }
%\author{\fnref{myfootnote}}
%\fntext[myfootnote]{Email address:m-khodabin@kiau.ac.ir}
%\author[]{\corref{mycorrespondingauthor}}
%%\ead{Email:Fatemeh.mahmoudi63@yahoo.com}
%% \address[label3]{m-khodabin@kiau.ac.ir}
%
%\author{}
%
%\address{}

\begin{abstract}
In this work, we will show the existence and uniqueness of the solution to semi-linear stochastic differential equations driven by weighted fractional Brownian motion with delay. We also prove smoothness of the density of the solution with respect to Lebesgue's measure on $\mathbb{R}^d$ for $d \geq 1$.
\end{abstract}


\end{frontmatter}
% main text
%%\linenumbers
%%\linenumbers
%%%%%%%%%%%%%%%%%%%%%%%%%%%%%%%%%%%%%%%%%%%%%%%%%%%%%%%
%%%%%%%%%%%%%%%%%%%%%%%%%%%%%%%%%%%%%%%%%%%%%%%%%%%%%%%
\section{Introduction}\label{intro}
 Theory of  stochastic differential equations (SDEs) driven by fractional Brownian motion (fBm) is widely studied using different approach (see for instance  \cite{Nua2002, Nua2003, NuaRas2002, 3, 4, 5}). These models have many applications in finance, telecommunications, image processing and turbulence; \cite{Nor95, Elwo, 10, Shi99}.\\
One of great interest area in study of fractional SDEs is on investigating the existence, uniqueness and regularity of the density of solutions to SDEs.
When $H > \frac12$, the existence and uniqueness of the solution are obtained by Lyons in \cite{Lyo94}, Zahle in \cite{Zah01}, Nualart-Rasc{a}nu in \cite{NuaRas2002} and by using Young's integration theory in \cite{BaH07}. The problem of existence and uniqueness of solution is considered in \cite{NeuNor08} for a Hurst parameter $H >1/3$, and 
extended to $H >1/4$ in \cite{Tin09}. 
Recently, when the drift is locally Lipschitz  and unbounded in the neighborhood of the
origin, particularly to the mean-reverting stochastic volatility models in finance, the existence and positivity of a unique solution have studied in \cite{zh20}.\\
In the presence of delay in SDEs, Wei and Wang \cite{Wei07} considered the problem of the existence and uniqueness of the solution to stochastic functional differential equations with infinite
delay. Caraballo et al. \cite{Cara11} have studied the existence, uniqueness and exponential asymptotic
behavior of mild solutions to stochastic delay evolution equations driven by a fractional
Brownian motion. In 2015, Boudaoui et al. \cite{BouCar15} showed the existence of mild solutions to stochastic impulsive evolution equations with time delays  via a new fixed point analysis approach. It is worth mentioning that in these two last atricles, the diffusion coefficient functions are dependent only on the parameter $t$ and are not dependent onthe state process. More studies have discussed  in\cite{FR06} and then in  \cite{Bou12} on stochastic delayed differential equations (SDDEs) with fractional Brownian motion when the diffusion functions depend upon the past value of the state process with time delay $\tau$  (actually, as $B=0$ in SDDE (\ref{s1})).   \\
The Malliavin calculus is a useful tool to study the regularity of the densities of the
solutions to SDEs. We refer the reader to \cite{2} and \cite{Nua1998} for more details of this
theory. The problem of smoothness of density of the solutions to stochastic differential equations driven by a fractional
Brownian motion with Hurst parameter greater than $1/2$ in the one-dimensional case is solved in \cite{NS06}
by using Doss-S\"{u}ssman methods and in \cite{NS05} and \cite{HN06} under ellipticity assumptions. 
Rovira and Ferrante \cite{FR06} established the existence and regularity of density of SDEs with delay via Young's integration. The authors considered a SDE with the diffusion coefficient function depends upon only the past value of solution with time delay; i.e,  it depends on $X_{t-\tau}$ for a solution $X(t)$ and time delay $\tau$. 
In 2012, the authors \cite{Tin12} have shown the existence of a $C^\infty$-density  to a general class of Young delay equations driven by fBm with Hurst parameter $H> 1/2$.\\
The weighted fractional Brownian motion (wfBm) is a general expression for the fractional Brownian motion introduced in Bojdecki et al. \cite{BoGo08}. Due to the memory effects of various phenomena in the real world (such as  a wide variety of natural and financial markets), it would be reasonable
to replace fBm by weighted fBm.  In this erea, a few studies can be found in some especial subjects; see for instance Bojdecki et al. \cite{BoGo07, BoGo082}, Garz\'{o}n \cite{Ga09}, Shen et al. \cite{Shen16, Shen20}, Yan et al. \cite{YaWa14}, and the references therein.\\ 
In current paper, we consider the following semi-linear SDDE of the form 
 \begin{align}\label{s1}
 dx(t)&=\Big\{Ax(t)+f( x(t-\tau))\Big\}dt+\Big\{B x(t)+\sigma(x(t-\tau))\Big\}\delta^{a,b}B^{a,b}(t),~~~t\in[0,T],\nonumber\\
 {x}(t)&=\xi_0(t), \quad  \qquad t \in [-\tau,0],
 \end{align}
where $A$  is  a $d\times d$-matrix  and the stochastic integral is a type of Skorohod integral with respect to weighted fBm.  We recall some basic elements of this stochastic calculus in section 3. \\
%...................................................................\\
 In the case of weighted fractional Brownian motion, to the best of our knowlege, there is no paper which study the existence and smoothness of the density of the solution to linear SDDE (\ref{s1}). 
We are interested to consider Malliavin differentiability and smoothness of the density of the solution to SDDE (\ref{s1}) under the usual globally Lipschitz conditions on coefficient functions. 
\begin{assumption}\label{l548}
 Let $K_0$ is the largest integer number that $K_0\tau < T$ and $M$ is an integer number such that $K_0+1 < 2M$.
\begin{itemize}
\item  Functions $f$ and $\sigma$ are continuous and $2M$-times differentiable, whose their derivatives are bounded with some constant $K_2$. 
\item There exist some positive constants $L$ and $K_1$ such that for every $ x,y \in \mathbb{R}^d $ and $ t,s\in[0,T]$
 \begin{equation*}
 \begin{split}
|f(y_1)-f(y_2)|^2 +|\sigma(y)-\sigma(x)|^2  &\leq K_1\vert y_1-y_2\vert^2 ,  \qquad \qquad
|f(y)|^2 +|\sigma(x)|^2 \leq L\left(1+\vert y \vert^2 \right), \qquad 
\end{split}
\end{equation*}
\end{itemize}
\end{assumption}
In addition, to obtain smoothness of the density, we also assume that the function $\sigma$ has a lower bound greater than zero, confirming a H\"{o}rmander's type condition. \\

  Our main approach to prove these problems is inspired from \cite{FR06}. More precisely, we shall construct the solution of SDDE (\ref{s1}) step by step within the intervals $[i\tau, (i+1)\tau]$, for any integer $i=1,\cdots, K_0$, and then in the interval $[K_0\tau, T]$. We successively determine a comprehensive exposition for its Malliavin derivatives thruoghout the steps. This exposition will allow us to achieve some upper bounds for Malliavin derivatives of the solution.  \\
In order to obtain these results, the need for suitable upper bound inequality on supremum of stochastic Skorokhod integral generated by wfBm is essential. This inequality is a general type of $L^p$-maximal inequality in \cite{33} on stochastic integral with respect to wfBm (Theorem \ref{l0}).\\
The organization of the article is as follows: In sections 2 and 3, we recall some basic elements of wighted fractional Brownian motion and Malliavin calculus based on this type of stochastic integral. In section 4, we will prove a version of maximal inequality to achieve main results. The existence and uniqueness and Malliavin differentiability of the solution will be established in section 5. Uniformly boundedness of the moments of the solution and its Malliavin deivatives are investigated in section 6. Finally, we discuss the problem of regularity of the density in section  7.
%%%%%%%%%%%%%%%%%%%%%%%%%%%%%%%%%%%%%%%%%%%%%%%%%%%%%%
\section{Weighted fractional Brownian motion}
For all parameters $ a,b$ with $ a> -1,~~ |b|< 1 $ and $|b|< a+1$, a weighted fractional Brownian motion, denoted by $ B^{a,b}$, on the complete probability space $ (\Omega,\textit{F}, P)$  is a mean zero Guassian process with simple covariance function\cite{Shen16}:
\begin{equation*}
R^{a,b}(t,s)=E[B^{a,b}(t)B^{a,b}(s)]= \int_0^{s \wedge t}u^{a}[(t-u)^{b}+(s-u)^{b}]du,~~~ s,t\geq 0.
\end{equation*}
 This process is  $\frac{a+b+1}{2}$-self-similar, long-range dependence with H$ \ddot{o} $lder paths. It is neither a semimartingale nor a Markov process if $ b\neq 0$, and for some constant $c_{a,b}$ and  for any $s,t\geq 0$
\begin{equation}\label{bfbm}
c_{a,b}(t\vee s)^a|t-s|^{1+b}\leq E\left[\left(B^{a,b}(t)-B^{a,b}(s)\right)^2\right]\leq K_{a,b}(t\vee s)^a|t-s|^{1+b},
\end{equation}
 where $ K_{a,b}=\frac{b}{2\mathbb{B}(a+1,b+1)}$. Some surveys and references could be found in \cite{29, 2, Shen16}.\\
 When $-1<b<1$ and $a=0$, covariance of the process coincides particularlly with the covariance of the fractional Brownian motion with Hurst index 
 $\frac{b+1}{2}$ \cite{28};  
\begin{equation*}
E[B^{a,b}(t)B^{a,b}(s)]=\frac{1}{b+1}[t^{b+1}+s^{b+1}-|s-t|^{b+1}].
\end{equation*}
Let $H$ be the Hilbert space defined as the closure of the linear space $ \mathcal{E} $ of  indicator functions $\{\mathbf{1}_{[0,t]}, t\in[0,T]\}$  with respect to the inner product
\begin{equation*}
\left\langle \mathbf{1}_{[0,s]}\mathbf{1}_{[0,t]}\right\rangle _{H} =R^{a,b}(t,s).
\end{equation*}
Consider the Gaussian processes $B^{a,b}(u)$ on $H$ such that for every $u_1,u_2 \in H$, $\mathbb{E}(B^{a,b}(u_1)B^{a,b}(u_2)) = \left\langle u_1,u_2\right\rangle _{H}$. 
In \cite{Shen16} have been shown that 
$\textrm{H}=\{ u : [0,T]\rightarrow \mathbb{R}; \| u\| _{\textrm{H}}< \infty \}$,
where 
\begin{equation}\label{inner}
\| u\|^2_{\textrm{H}} =\int_{0}^{T}\int_0^T u(s)u(t)\phi(t,s)dt ds :=\int_{0}^{T}\int_0^T u(s)u(t)b(t\wedge s)^a (t\vee s-t\wedge s)^{b-1}dt ds,
\end{equation}
 The map $u  \in \mathcal{E} \rightarrow B^{a,b}(u)$ is an isometry from $ \mathcal{E} $ to the Gaussian space generated by $ B^{a,b}$ and it can be extended to $ \textrm{H}$.
We should note that The subspace $|\textrm{H}|$ of measurable functions $u$  equipped by the norm 
\begin{equation}\label{lk}
\| u\|^2_{|\textrm{H}|} =\int_0^T\int_0^T |u(s)||u(t)|\phi(t,s)dt ds<\infty,
\end{equation}
 is a Banach space and $\mathcal{E}$ is dense in $|\textrm{H}|$. Moreover, Pipiras and Taqqu \cite{32} have shown that as $a+b < 1$
\begin{equation}\label{sub1}
  L^2([0,T])\subset L^{2/(a+b+1)}\subset |\textrm{H}| \subset \textrm{H}.
\end{equation}
%%%%%%%%%%%%%%%%%%%%%%%%%%%%%%%%%%%%%%%%%%%%%5
\section{Preliminaries on Malliavin calculus}
We briefly recall some Malliavin criteria on fractional Brownian motion and wieghted fractional Brownian motion in \cite{2, Duncan, Shen16}.
When $b>0$, we denote by $ \mathcal{S} $  the set of smooth functionals of the form
 \begin{equation*}
F=f\left( B^{a,b}(u_{1}),B^{a,b}(u_{2}),\cdots B^{a,b}(u_{n})\right),
\end{equation*}
 where $ f \in C^{\infty}_{b} (\mathbb{R}^n)$ (f and all its derivatives are bounded) and $ u_{i} \in \textrm{H},~ i=1,2,\cdots ,n$. For every $F \in \mathcal{S}$, define
\begin{equation*}
D^{a,b}F=\sum_{i=0}^{n} \frac{\partial f}{\partial x_{i}}\left( B^{a,b}(u_{1}),B^{a,b}(u_{2}),\cdots B^{a,b}(u_{n})\right)u_{i}.
\end{equation*}
 The derivative operator $D^{a,b}$ is a closable operator from $L^{p}(\Omega)$ into $L^{p}(\Omega,\textrm{H})$ for every $p \geq 1$. We denote by $ \mathbb{D}^{1,p}$ the closure of $ \mathcal{S} $ respect to the norm
\begin{equation*}
\|F\|_{1,p}^p=E|F|^p+E\|D^{a,b}F\|^p_H.
\end{equation*}
For any $k \geq 1$, $k$-times iteration of  the derivative operator is expressed by $D^{a,b} \ldots D^{a,b}F$ ($k$-times). For given Hilbert space $V$, the corresponding sobolev space $ \mathbb{D}^{1,p}(V)$ is the domain of the derivative operator of $V$-valued random variable.\\
The adjoint of the derivative operator, called $ \delta^{a,b}( \varphi )$, is characterized by the duality relationship
\begin{equation*}
E\left[F\delta^{a,b}(\varphi)\right]=E\left\langle D^{a,b}F,\varphi\right\rangle_H,
\end{equation*}
for any $F \in \mathcal{S}$. 
For every stochastic process $ \varphi \in \mathbb{D}^{1,2}(|\textrm{H}|)$ 
indefinite Skorohod integral is expressed as
\begin{equation*}
\delta^{a,b}(\varphi)=\int _{0}^T \varphi(s)\delta^{a,b} B^{a,b}(s).
\end{equation*}
In a similar argument of \cite{Shen16, 33}
\begin{equation}\label{sdelta}
\int _0^T \varphi(s)dB^{a,b}(s)= \int _{0}^T \varphi(s)\delta^{a,b} B^{a,b}(s)+b\int_0^T\int_0^T D^{a,b}_t \varphi(s) (t\wedge s)^a (t\vee s-t\wedge s)^{b-1}dt ds,
\end{equation}
providing the second summand is finite. A classical relation is that, 
\begin{align}
\mathbb{E}\left[\vert\delta^{a,b}(\varphi)\vert^2\right]&= \mathbb{E}\|\varphi\|^{2}_{H}+ \mathbb{E}\int_{[0,T]^4}D^{a,b}_\xi \varphi(r)D^{a,b}(\eta) \varphi(s) \phi(\eta,r)\phi(\xi,s)ds dr d\xi d\eta \nonumber\\
&\leq\mathbb{E}\|\varphi\|^{2}_{|H|}+ \mathbb{E}\int_{[0,T]^4}|D^{a,b}_\xi \varphi(r)||D^{a,b}_\eta \varphi(s)||\phi(\eta,r)||\phi(\xi,s)|ds dr d\xi d\eta.  \label{tavan2}
\end{align}
For any $p> 1$ we denote by $ \mathbb{L}^{1,p}_{(a+b+1)/2}$ the set of processes $\varphi$ in $ \mathbb{D}^{1,2}(|\textrm{H}|)$ such that
\begin{equation*}
\| \varphi\|^p_{\mathbb{L}^{1,p}_{(a+b+1)/2}}= \mathbb{E}\| \varphi\|^p_{{L}^{2/(a+b+1)}([0,T])}+\mathbb{E}\|D^{a,b} \varphi\|^p_{{L}^{2/(a+b+1)}([0,T]^2)}<\infty.
\end{equation*}
In view of (\ref{sub1}) and (\ref{tavan2}), one can easily show that for every $\varphi \in \mathbb{L}_{(a+b+1)/2}^{1,p}$
 \begin{align}\label{02}
 \mathbb{E}\left[|\delta ^{a,b}(\varphi)|^p\right] \leq C_{(a+b+1)/2,p}\left(\|\mathbb{E}\varphi\|_{L^{2/(a+b+1)}([0,T])}^p+\mathbb{E}\|D^{a,b}\varphi\|_{L^{2/(a+b+1)}([0,T]^2)}^p\right)
 \end{align}
%%%%%%%%%%%%%%%%%%%%%%%%%%%%%%%%%%%%%%%%%%%%%%%%%%%%%%%
\section{$L^p$-maximum estimation of Skorokhod integral driven by weighted fractional integrals}
In this section, we shall establish $ L^p$-maximal estimation of the divergence process driven by weighted fractional Brownian motion for every $1<p_0<\infty$. To do this, we will apply some upper bound inequality  in \cite{34} to Hardy type operator $ T_{\alpha,\beta}$ of any function $f \in L^{p_0}\Big((a_1,b_1);v\Big)$ defined by
 \begin{align*}
 T_{\alpha,\beta}f(x)=\int_{a_1}^{x} \frac{u(s)W^{\beta}(s)f(s)w(s)ds}{\left(W(x)-W(s)\right)^{1-\alpha}}, \qquad x\in[0,T],
 \end{align*}
where $W$ is a non-negative, strictly increasing and locally absolutely continuous function on interval $I=(a_1,b_1)$  and $\frac{dW(s)}{ds}=w(s)$, and also $u$ is almost everywhere positive locally integrable function as well as $v$ is positive kernel function. \\
 Let $ \frac{1}{p_0}+\frac{1}{p'}=1$ and for every $p_0 \leq q <\infty $ denote 
\begin{equation*}
A_{\alpha,\beta}= \sup_{z \in I} \Big(\int_{a_1}^{z} u^{p'}(s) W^{p'\beta}(s)w(s)ds\Big)^{\frac{1}{p_0}}\Big(\int_{z}^{b_1} W^{q(\alpha-1)}(s)ds\Big)^{\frac1q}.
\end{equation*}
In Theorem 3.3 of \cite{34} have been shown that for $0 <\alpha <1$ and $\beta \geq 0$ there exists some constant $C$ such that  
 \begin{align}\label{01}
  \left(\int_{a_1}^{b_1}\left(T_{\alpha,\beta}f(x)\right)^{q}v(x)dx\right)^{\frac{1}{q}} \leq C \left(\int_{a_1}^{b_1}\left( f(x)\right)^{p_0}w(x)dx\right)^{\frac{1}{p_0}}.
 \end{align}
if and only if $A_{\alpha,\beta} < \infty$.
Now, we establish a $L^{p}$-maximal estimation for the indefinite integral $ \int_{0}^{t} \varphi(s)\delta^{a,b} B^{a,b}(s)$. 
 %%%%%%%%%%%%%%%%%%%%%%%%%%%%%%%%%%%%%%%%%%%%%%%%%%%%%%%%%%%%%%%%%%%%%%%%%%%%%%%%
 \begin{theorem}\label{l0}
For $ a>-1,~ |b|<1$ such that $a+b <1$, let $ u=\{u(t) , t\in [0,T]\}$  be a stochastic process in
$\mathbb{L}^{1,p}_{(a+b+1)/2}$. Then for every $p>\frac{4}{a+b+1}$, there exists some constant $C_1$ such that
\begin{align*}
\mathbb{E}\left(\sup_{t \in [0,T]}\left|\int_{0}^{t}u(s)\delta^{a,b} B^{a,b}_{s}\right|^{p} \right)
&\leq C_1 T \bigg\{\left(\int_{0}^{T}\left|\mathbb{E}u(s)\right|^{\frac{2p}{p(a+b+1)-2}} ds\right)^{\frac{p(a+b+1)-2}{2}}\\
 &+\mathbb{E}\left(\int_0^T\left(\int_0^T\left|D_{s}^{a,b}u(r)\right|^{2/(a+b+1)}dr\right)^{\frac{p(a+b+1)}{p(a+b+1)-2}}ds\right)^{\frac{p(a+b+1)-2}{2}}\bigg\},\\
\end{align*}
where the constant $ C_1 $ depends on $(a+b+1)/2$, $p$ and $ T$.
\end{theorem}
%%%%%%%%%%%%%%%%%%%%%%%%%%%%%%%%%%%%%%%%%%%%%%%%%%%%%%%%%%%%%%%%%%%%%%%%%%%%%
\begin{proof} 
The proof is motivated by the proof of Theorem 4 in \cite{33}. Set a positive constant $\lambda$ such that 
$\frac{a+b+1}{2} < \lambda < 1$ and put $a_0= a+b+1$ which is a positive number. Now, using the equality
 \begin{align*}
c_{a,b,\lambda}=\int _{r}^{t} t^{\lambda}(t-\theta)^{-a_0}r^{a_0} (\theta-r)^{\lambda-1}d\theta<\infty,
\end{align*}
we know that 
\begin{align*}
\int _{0}^{t}u(s)\delta^{a,b} B^{a,b}(s)=c_{a,b,\lambda}^{-1}\int _{0}^{t} u(s) \left(\int_{s}^{t}t^{\lambda} (t-r)^{-a_0}s^{a_0}(r-s)^{\lambda-1}dr\right)\delta^{a,b} B^{a,b}(s).
\end{align*}
One can apply Fubini's stochastic theorem to result
\begin{align*}
\int _{0}^{t}u(s)\delta^{a,b} B^{a,b}(s)=c_{a,b}^{-1}\int _{0}^{t}t^{\lambda} (t-r)^{-a_0} \left(\int_{0}^{r} u(s)s^{a_0}(r-s)^{\lambda-1}\delta^{a,b} B^{a,b}(s)\right)dr.
\end{align*}
From H$ \ddot{o} $lder's inequality for some other constant c
\begin{align*}
\left|\int _{0}^{t}u(s)\delta^{a,b} B^{a,b}(s)\right|^p\leq  c^{{-1}}\int _{0}^{t}\left|\int_{0}^{r} u(s)s^{a_0}(r-s)^{\lambda-1}\delta^{a,b} B^{a,b}(s)\right|^p dr.
\end{align*}
Therefore,
\begin{align*}
\mathbb{E}\left(\sup_{t \in [0,T]}\left|\int _{0}^{t}u(s)\delta^{a,b} B^{a,b}(s)\right|^{p} \right)&\leq
c^{-1}\mathbb{E} \int _{0}^{T}\left|\int_{0}^{r} u(s)s^{a_0}(r-s)^{\lambda-1}\delta^{a,b} B^{a,b}(s)\right|^p dr.
\end{align*}
Using now inequality (\ref{02}) we obtain
\begin{align*}
\mathbb{E}\left(\sup_{t \in [0,T]}\left|\int _{0}^{t}u(s)\delta^{a,b} B^{a,b}(s)\right|^{p} \right)&\leq c^{-1}C_{\frac{a+b+1}{2}}\bigg\{ \int_{0}^{T}\left(\int_{0}^{r}s^{2}(r-s)^{2(\lambda-1)/a_0}\left|\mathbb{E}u(s)\right|^{2/a_0} ds\right)^{pa_0/2}dr \nonumber\\
&+ \mathbb{E} \int_0^T \left( \int _0^r \int_0^T s^{2}(r-s)^{2(\lambda-1)/a_0} \left|D^{a,b}_\theta u(s)\right|^{2/a_0}d\theta ds\right)^{pa_0/2}dr\bigg\}\nonumber\\
&=: c^{-1}C_{a_0/2}(I_{1}+I_{2}).\\
\end{align*}
To bound the terms  $I_{1}$ and $I_{2}$, we apply the Hardy type operator inequality (\ref{01}) for $q=p\frac{(a+b+1)}{2}$, $p_{0}=\frac{pa_0}{pa_0-2}$, $\alpha=1+2\frac{(\lambda-1)}{a_0}$, $\beta=2$, $u(.)=1$, $v(.)=\frac{1}{T}$ and $W(s)=s$. Under our assumptions and the way of choosing $\lambda$, can has  $1 < p_0 < q$ and $0 <\alpha <1$ as well as $A_{\alpha,\beta}< \infty$. Consequently,
\begin{align*}
I_{1}&\leq CT \left(\int_{0}^{T}\left|\mathbb{E}u(r)\right|^{\frac{2p_0}{a_0}} dr\right)^{\frac{q}{p_0}},
\end{align*}
\begin{align*}
I_{2}&\leq
 C \mathbb{E}\left(\int_0^T \left(\int_0^T \left |D_{s}^{a,b}u(r)\right|^{2/a_0}ds\right)^{p_0}dr\right)^{\frac{q}{p_0}},
\end{align*}
which completes the proof.
\end{proof}
According to the proof of this theorem, for any partition $0 \leq t_1 \cdots \leq t_N=T$ and every $0 \leq k \leq N$, the following corollary can be covered, if the weighted function $v$ in the proof of Theorem \ref{l0} replaced by $v=\frac{1}{t_{k+1}-t_k}$. 
\begin{corollary}\label{corr1}
Under conditions of Theorem \ref{l0}, for every $0 \leq k \leq N$ and any partition $0 \leq t_1 \cdots \leq t_N=T$, there exists some constant which we denote again $C_1$ such that 
\begin{align*}
\mathbb{E}\left(\sup_{t \in [t_k,t_{k=1}]}\left|\int_{t_k}^{t}u(s)\delta^{a,b} B^{a,b}_{s}\right|^{p} \right)
&\leq C_1 (t_{k+1}-t_k)\bigg\{\left(\int_{t_k}^{t_{k+1}}\left|\mathbb{E}u(s)\right|^{\frac{2p}{p(a+b+1)-2}} ds\right)^{\frac{p(a+b+1)-2}{2}}\\
 &+\mathbb{E}\left(\int_{t_k}^{t_{k+1}}\left(\int_{t_k}^{t_{k+1}}\left|D_{s}^{a,b}u(r)\right|^{2/(a+b+1)}dr\right)^{\frac{p(a+b+1)}{p(a+b+1)-2}}ds\right)^{\frac{p(a+b+1)-2}{2}}\bigg\},\\
\end{align*}
\end{corollary}
\begin{corollary}\label{l54} When the function $ u(.)$  is a deterministic function, it is obviously that $ D_{r}u(s)=0$ and therefore for every $p > 4/(a+b+1)$
\begin{equation}\label{1co}
\begin{split}
\mathbb{E}\left(\sup_{t \in [0,T]}\left|\int _{0}^{t}u(s)\delta^{a,d}B^{a,b}(s)\right|^{p}\right) &\leq C_1T \left(\int_{0}^{T}\left|u(s)\right|^{\frac{2p}{p(a+b+1)-2}} ds\right)^{\frac{p(a+b+1)-2}{2}},
\end{split}
\end{equation}
and using Young's inequality implies that
\begin{equation*}
\mathbb{E}\left(\sup_{t \in [0,T]}\left|\int _{0}^{t}u(s)\delta^{a,b}B^{a,b}(s)\right|^{p}\right) \leq C_1T^2 \left(\int_{0}^{T}\left|u(s)\right|^{p}ds\right).
\end{equation*}
\end{corollary}

%%%%%%%%%%%%%%%%%%%%%%%%%%%%%%%%%%%%%%%%%%%%%%
\section{Existence and Malliavin differentiability of the solution}
As an analythic level, fisrt we show the existence of the solution of Equation (\ref{s1}) step by step in the intervals $[i\tau, (i+1)\tau]$ for any integer $i=1,\cdots, K_0$, and then in the interval $[K_0\tau, T]$. Malliavin differentiability of the solution in the sense of stochastic Skorokhod integral will be concluded, recursively. The uniqueness of the solution can be proved throughout any step.\\
We assume the following conditions on the function $\xi_0(.)$ and also recall two proposition from \cite{Duncan} for future use. 
\begin{assumption}\label{assum2}
The measurable function  $\xi_0(.)$ is Malliavin Differentiable up to the order $2M+1$ and for every $p \geq 1$, there exists some constant $C_{\xi,p}$ such that
\begin{equation}
\mathbb{E}\Big(\sup_{0 \leq r \leq \tau}\xi_0(r-\tau)^{p}\Big) \leq C_{\xi,p} < \infty,
\end{equation} 
\end{assumption}
 Let us recall the Ito's formula has been introduced in \cite{Duncan} which is essential in our modification. Proceeding the proof of Ito's formula therein show that it would be also hold in weighted stochastic integral with  inner product $<.,.>_H$ defined in (\ref{inner}) with the kernel function $\phi(.,.)$, as the authors have mentioned in their paper.  So, we shall rewrite Ito's formula for wfBm instead of fBm. Recall $\mathcal{L}(0,T)$ as the set of Malliavin differentiable stochastic processes $G$ such that 
 $\mathbb{E}\| G\|_{H}+\mathbb{E}\|D^{a,b} G\|_{H\otimes H}<\infty$, and for any sequence of partiotion $\pi:0=t^n_0 \leq \cdots \leq t_n^n=T$ of $[0,T]$ such that $\vert \pi \vert \rightarrow 0$ as $n \rightarrow \infty$
\begin{equation*}
\sum_{i=0}^n \mathbb{E}\Big\{ \int_{t_{i}^n}^{t_{i+1}^n}\int_0^T( D_{r}^{a,b}G^\pi_{t_i^n} -D_{r}^{a,b} G_s)\phi(r,s) dr ds\Big\}^2+ \mathbb{E}\Big\{ \| G^\pi - G \|_H \Big\} < \infty
\end{equation*}
\begin{proposition}(Duncan \cite{Duncan}, Theorem 4.3)\label{ito}
Let $\{F_s, s \in [0,T]\} \in \mathcal{L}(0,T)$ be a stochastic process such that for some $\alpha_3 > 1-\frac{b+1}{2}$, 
$$\mathbb{E}\Big(\vert F_{s_1}- F_{s_2} \vert^2\Big) \leq c_{F} \vert s_1- s_2\vert^{2\alpha_3}$$  
where $\vert s_1-s_2 \vert \leq \delta$ for some $\delta >0$ and 
$$\lim_{\vert s_1-s_2\vert \rightarrow 0}\mathbb{E}\Big(  \vert (D_{s_1}^{a,b})^\phi ( F_{s_1}- F_{s_2}) \vert^2\Big)=0.$$  
in which $(D_{s}^{a,b})^\phi  F = \int_0^T \int_0^T D_{r}^{a,b} F .\phi(r,s)dr $.
Assum that $\mathbb{E}(\sup_{0 \leq s \leq T} \vert G_s \vert) < \infty$ and $f:\mathbb{R}_{+}\times \mathbb{R} \longrightarrow \mathbb{R}$
 is a function in 
$\mathbb{C}_{b}^{1,2}(\mathbb{R}_{+}\times\mathbb{R})$.
 If $\eta_t = \xi +\int_0^t G_s ds + \int_0^t F_s \delta^{a,b}B_s^{a,b}, ~ \xi \in \mathbb{R}$ for $t \in [0,T]$ and $\frac{\partial f}{\partial x} (s,\eta_s)F_s \in \mathcal{L}(0,T)$. Then for $t \in [0,T]$
\begin{align*}
f(t, \eta_t)&=f(0,\xi)+ \int_0^t \frac{\partial f}{\partial t} (s,\eta_s) ds + \sum_{k=1}^{n}\int_0^t \frac{\partial f}{\partial x} (s,\eta_s) G_s ds\\
& +\int_0^t \frac{\partial f}{\partial x} (s,\eta_s)F_s  \delta^{a,b}B_s^{a,b} + \int_0^t \frac{\partial^2 f}{\partial x^2} (s,\eta_s)F_s (D_s^{a,b})^\phi \eta_s ds, 
\end{align*}
\end{proposition} 
\begin{proposition}(Duncan \cite{Duncan}, Theorem 4.6)\label{ito1}
Let $\{F^i_s, s \in [0,T]\} \in \mathcal{L}(0,T)$ and the function $f$ satisfy the conditions of Proposition \ref{ito}. Assume
$\eta^k_t = \xi_k +\int_0^t G^k_s ds + \int_0^t F^k_s \delta^{a,b}B_s^{a,b}, ~ \xi \in \mathbb{R^n}$ for $t \in [0,T]$ and $\frac{\partial f}{\partial x_k} (s,\eta_s)F^k_s \in \mathcal{L}(0,T)$. Then for $t \in [0,T]$
\begin{align*}
f(t, \eta^1_t, \cdots,  \eta^n_t)&=f(0,\xi_1, \cdots, \xi_n)+ \int_0^t \frac{\partial f}{\partial t} (s,\eta_s) ds + \sum_{k=1}^{n}\int_0^t \frac{\partial f}{\partial x_k} (s,\eta_s) G^k_s ds\\
& +\sum_{k=1}^{n}\int_0^t \frac{\partial f}{\partial x_k} (s,\eta_s)F^k_s  \delta^{a,b}B_s^{a,b} +
\sum_{k,l=1}^{n}\int_0^t \frac{\partial^2 f}{\partial x_k \partial x_l} (s,\eta_s)F^k_s (D_s^{a,b})^\phi \eta^l_s ds,
\end{align*}
\end{proposition} 
It is worth mentioning that, Theorem 4.2. in \cite{Duncan} can be also rewritten in the sense of weighted fractional Brownian motion, clearly.
\begin{proposition}\label{dfphi}
If $\Big\{F_s, s \in [0,T]\Big\} \in \mathcal{L}(0,T)$ and $\sup_{0 \leq s \leq T} \mathbb{E}\Big(\vert(D_s^{a,b})^\phi F_s \vert^2\Big) < \infty$. Then for $s,t \in [0,T]$
\begin{equation*}
(D_s^{a,b})^\phi \left\{\int_0^t F_u \delta^{a,b}B_u^{a,b}\right\} =\int_0^t   (D_s^{a,b})^\phi  F_u \delta^{a,b}B_u^{a,b}+ \int_0^t F_u \phi(s,u) du, \quad a.s.   
\end{equation*}
\end{proposition}
Now, consider the following linear SDE driven by weighted fractional Brownian motion
\begin{align}\label{psiuni}
 d\psi(t)&=A\psi(t)dt+B \psi(t)\delta^{a,b}B^{a,b}(t),~ as~~t\in[0,T],  \quad and \quad \psi(0)=1.
 \end{align}
To show main results, we start with showing that, the SDE (\ref{psiuni}) has a unique solution $\psi(.)$ with an exponential exposition and it has Malliavin derivatives which can be presented as a function of the solution $\psi(.)$.\\
Thanks to  Proposition \ref{ito}, following as in the proof of Theorem 2.5. and Lemma 2.2 in \cite{flinear},  one can coclude that the solution of (\ref{psiuni}) is $$\psi(t)=exp\Big\{At+B B^{a,b}(t)-\frac12 B^2\int_0^t\int_0^t \phi(s,s')dsds'\Big\}$$ and  $D^{a,b}_r\psi(t)=\psi(t)B1_{0 \leq r \leq t}$ and equivalently $(D^{a,b}_r)^\phi \psi(t)=\psi(t)B \int_0^t \phi(r,s)ds$. This fact leads to the conclusion that, the process $\psi^{-1}$ is the solution to the SDE
\begin{align*}
 d\psi^{-1}(t)&=\Big(-A+ B^2\int_0^t\int_0^t \phi(s,s')dsds'\Big)\psi^{-1}(t)dt-B \psi^{-1}(t)\delta^{a,b}B^{a,b}(t),~~~t\in[0,T],\\
 \psi^{-1}(0)&=1,
\end{align*}
 and $D^{a,b}_r\psi^{-1}(t)=-\psi^{-1}(t)B1_{0 \leq r \leq t}$, which is equivalent to $(D^{a,b}_r)^\phi \psi^{-1}(t)=-\psi^{-1}(t)B \int_0^t \phi(r,s)ds$.
\begin{theorem}\label{psi}
The unique solution $\psi(.)$ has uniformly bounded moments; i.e.,  there exists some positive constant $C_{p}$ such that 
\begin{equation*}
\mathbb{E}\Big(\sup_{0 \leq t \leq T} \vert \psi(t) \vert^p  \Big) \leq C_{p}
\end{equation*}
\end{theorem}
\begin{proof} 
According to the relation $D^{a,b}_r\psi(t)=\psi(t)B1_{0 \leq r \leq t}$ and inequality (\ref{tavan2}) in connection with inclusion (\ref{sub1}) and then Gronwall's inequality, the uniquenss of the solution $\psi$ is resulted. The boundedness of the $p$-momoents of $\psi(.)$ follows from Theorem 3.3 in \cite{Duncan} and Proposition \ref{ito}), and deduces that there exists some constant $C_p$ such that 
\begin{equation*}
\mathbb{E}\Big(\sup_{0 \leq t \leq T} \vert \psi(t) \vert^p  \Big) \leq exp^{\Big\{pAT+\frac12(p^2-p)\int_0^T\int_0^T \phi(s,s')ds ds'\Big\} }\leq C_p.
\end{equation*}
\end{proof}
\begin{proposition}\label{boundpsi}
For every $p \geq 1$, there exists some constant $C_{\psi,p}$ such that 
\begin{equation*}
\mathbb{E}\Big(\sup_{0 \leq r \leq T}\vert \psi(r)^{-1}\vert^p\Big)+\mathbb{E}\left(\int_{0}^{T}\int_{0}^{T}\left|D_{s}^{a,b}\psi^{-1}(r)\right|^p drds \right)\leq C_{\psi,p} < \infty,
\end{equation*}
\end{proposition}
 As a consequence, the stochastic processes $\psi(.)$ and $\psi^{-1}(.)$ are in $\mathcal{L}(0,T)$ and satisfy the condition of Theorem \ref{ito}  with $\alpha_3=(1+b)/2$, applying Equations (\ref{tavan2}) and (\ref{bfbm}).\\
Now, we are ready to construct the solution of  Equation  (\ref{s1}) in the following theorem and obtain its Malliavin derivative. 
\begin{theorem}\label{exists}
Under Assumption \ref{l548} and the first part of Assumption \ref{assum2},  SDE (\ref{s1}) admits a unique solution on $[-\tau, T]$ which is also Malliavin differentiable up to the order $2M$.
\end{theorem} 
\begin{proof}
We will prove the assertion in four steps for the convenience of readers. The uniqueness of the solution throughout any step stands on a similar proof of Theorem \ref{psi} to show the uniqueness of $\psi(.)$ to SDE (\ref{psiuni}).\\
{\bf step 1}.  For $t \in [0,\tau]$, define the Gaussian process $$M_0(t)=\int_0^t f(\xi_0(s-\tau)) ds-B\int_0^t\int_0^s \sigma(\xi_0(s-\tau))\phi(u,s)du ds +\int_0^t \sigma(\xi_0(s-\tau) )\delta^{a,b}B^{a,b}(s)$$ and consider the following stochastic differential equation
 \begin{equation*}
 X_0(t)= \xi_0(0) + \int_0^t AX_0(s) ds + \int_0^t BX_0(s) \delta^{a,b}B^{a,b}(t)+ M_0(t)\quad t \in [0,\tau]
 \end{equation*}
To introduce the solution $X_0$, we first denote $Z_0(t):= \xi_0(0) + \int_0^t \psi^{-1}(s) dM_0(s)$. In view of Proposition \ref{dfphi}, since $\psi^{-1} \in \mathcal{L}[0,T]$  and for every $-\tau \leq s < 0$ and $r \geq 0$, we know $D_r^{a,b} \xi_0(s)=0$, therefore for any $0 \leq r \leq t$
\begin{align*}
D_r^{a,b}Z_0(t) &=\psi^{-1}(r)\sigma\Big(\xi_0(r-\tau)\Big)-\int_r^t  B\psi^{-1}(s)\left(f(\xi_0(s-\tau))-\int_0^s \sigma(\xi_0(s-\tau))\phi(u,s)du\right) ds \\
& -\int_r^t B
\psi^{-1}(s)\sigma(\xi_0(s-\tau))\delta^{a,b}B^{a,b}(t) \\
& = \psi^{-1}(r)\sigma(\xi_0(r-\tau)) - BZ_0(t)+BZ_0(r).
\end{align*}
It deduce that $D_t^{a,b}Z_0(t)= \psi^{-1}(t)\sigma(\xi_0(t-\tau))$, or equivalently  $$(D_t^{a,b})^\phi Z_0(t)= \psi^{-1}(t)\int_0^t \sigma(\xi_0(u-\tau))\phi(u,t) du.$$
Applying Ito's formula (\ref{ito1}) for $U_0(t):=e^{-B\int_0^t \int_0^s \phi(s,u)du ds}\psi(t)Z_0(t)$ and then substituting $(D_t^{a,b})^\phi\psi(t)$ and $(D_t^{a,b})^\phi Z_0(t)$ result 
\begin{align*}
dU_0(t) &=e^{-B\int_0^t \int_0^s \phi(s,u)du ds}\Big\{B(\int_0^t \phi(t,u)du )\psi(t)Z_0(t) dt + A\psi(t)Z_0(t) dt 
+ B\psi(t)Z_0(t) \delta^{a,b}B^{a,b}(t) \nonumber\\
&+  \psi(t) \psi^{-1}(t)dM_0(t) + (D_t^{a,b})^\phi\psi(t). \psi^{-1}(t) \sigma(\xi_0(t-\tau)) dt+  B\psi(t)(D_t^{a,b})^\phi Z_0(t) dt \Big\}  \nonumber \\
& = \Big\{ AU_0(t) +f(\xi_0(s-\tau))\Big\} dt + \Big\{ BU_0(t) +  \sigma(\xi_0(s-\tau) )\Big\}\delta^{a,b}B^{a,b}(t).
\end{align*}
It  show that for every $0\leq t \leq \tau$, the process $e^{-B\int_0^t \int_0^s \phi(s,u)du ds}\psi(t)Z_0(t)$ is a solution to SDE (\ref{s1}). Next, due to Proposition \ref{dfphi} and the relationship of $ D_r^{a,b}$ and $ (D_t^{a,b})^\phi$ in Proposition \ref{ito}, one deduce that $x(t)$ in the time interval $[0, \tau]$ has a weak derivative satisfying 
\begin{equation*}
 D_r^{a,b} X_0(t)=B X_0(r)+\int_r^t A D_r^{a,b} X_0(s)ds +\int_r^t B D_r^{a,b} X_0(s)\delta^{a,b}B^{a,b}(s),~~~t\in[0,\tau],
\end{equation*}
for every $0 <  r \leq t$. \\
{\bf step 2}. For every $k=1, \cdots, K_0-1$ and for all  $k\tau \leq t \leq (k+1)\tau$, define again the Gaussian processes $$M_k(t)=\int_{k\tau}^t f(X_k(s-\tau)ds -B\int_{k\tau}^t\int_0^s \sigma(X_k(s-\tau))\phi(u,s)du ds +\int_{k\tau}^t \sigma(X_k(s-\tau) )\delta^{a,b}B^{a,b}(t),$$ and consider the following stochastic differential equations
 \begin{equation*}
 X_{k}(t)=  X_{k-1}(k\tau) + \int_{k\tau}^t AX_{k}(s) ds + \int_{k\tau}^t BX_{k}(s) \delta^{a,b}B^{a,b}(t)+ M_{k-1}(t),  \quad t \in [k\tau,(k+1)\tau].
 \end{equation*}
similarly, if we define $Z_k(t):= X_{k-1}(k\tau) + \int_{k\tau}^t \psi^{-1}(s) dM_{k-1}(s)$, as $\psi^{-1} \in \mathcal{L}[0,T]$  and for every $-k\tau \leq s < (k+1)\tau$ and $r \geq k\tau$ we know $D_r^{a,b} X_{k-1}(s)=0$, then for any $k\tau \leq r \leq t$
\begin{align*}
D_r^{a,b}Z_k(t) &=0-\int_r^t B\psi^{-1}(s)\left(f(X_{k-1}(s-\tau))-\int_0^s \sigma(X_{k-1}(s-\tau))\phi(u,s)du\right) ds \\
& -\int_r^t B
\psi^{-1}(s)\sigma(X_{k-1}(s-\tau))\delta^{a,b}B^{a,b}(t) +\psi^{-1}(r)\sigma(X_{k-1}(r-\tau))\\
& = \psi^{-1}(r)\sigma(X_{k-1}(r-\tau)) - BZ_k(t)+BZ_k(r).
\end{align*}
Clearly, it deduce that $D_t^{a,b}Z_k(t)= \psi^{-1}(t)\sigma(X_{k-1}(t-\tau))$, or equivalently 
\begin{equation}\label{dphi}
(D_t^{a,b})^\phi Z_k(t)= \psi^{-1}(t)\int_0^t \sigma(X_{k-1}(u-\tau))\phi(u,t) du.
\end{equation}
 We employ Ito's formula (\ref{ito1}) for $U_k:=e^{-B\int_0^t \int_0^s \phi(s,u)du ds}\psi(t)Z_k(t)$ and substitute (\ref{dphi}) and get that  
\begin{align*}
dU_k(t) &=e^{-B\int_0^t \int_0^s \phi(s,u)du ds}\Big\{B(\int_0^t \phi(t,u)du )\psi(t)Z_k(t) dt + A\psi(t)Z_k(t) dt + B\psi(t)Z_k(t) \delta^{a,b}B^{a,b}(t) \\
&+  \psi(t) \psi^{-1}(t)dM_{k-1}(t) + (D_t^{a,b})^\phi\psi(t). \psi^{-1}(t) \sigma(X_{k-1}(t-\tau)) dt+  B\psi(t)(D_t^{a,b})^\phi Z_k(t) dt \Big\}  \\
& = \Big\{ U_k(t) +f((X_{k-1}(t-\tau))\Big\} dt + \Big\{ BU_k(t) +  \sigma(X_{k-1}(t-\tau) )\Big\}\delta^{a,b}B^{a,b}(t).
\end{align*}
 Thus SDE (\ref{s1}) has the solution  $e^{-B\int_0^t \int_0^s \phi(s,u)du ds}\psi(t)Z_k(t)$ for every  $k\tau \leq t \leq (k+1)\tau$. Next, Proposition \ref{dfphi} results this solution has also a Malliavin derivative satisfying 
\begin{align}
 D_r^{a,b} X_{k}(t)&=D_r^{a,b}X_{k-1}(k\tau)+ B X_{k}(r)+\sigma(X_{k-1} (r-\tau))1_{k\tau \leq r \leq t-\tau}\nonumber \\
&+\int_{k\tau\vee r}^t\left(A D_r^{a,b} X_{k}(s)+f'(X_{k-1}(s-\tau))D_r^{a,b}X_{k-1}(s-\tau)1_{ r\leq s-\tau}\right)ds \nonumber\\
&+\int_{k\tau \vee r}^t\left(B D_r^{a,b} X_{k}(s)+\sigma'(X_{k-1}(s-\tau))D_r^{a,b}X_{k-1}(s-\tau)1_{ r\leq s-\tau}\right)\delta^{a,b}B^{a,b}(s),~~~t\in[k\tau ,(k+1)\tau], \label{d1k}
 \end{align}
for every $0 \leq r \leq t-\tau$ and also
\begin{equation}\label{d1kd}
 D_r^{a,b} X_{k}(t)=B X_{k}(r)+\int_{k\tau \vee r}^t A D_r^{a,b} X_{k}(s)ds +\int_{k\tau \vee r}^t B D_r^{a,b} X_{k}(s)\delta^{a,b}B^{a,b}(s),~~~t\in[k\tau ,(k+1)\tau],
\end{equation}
for every $t-\tau <  r \leq t$.\\
{\bf Step 3}.  Perform step 2 for $k=K_0$ and result that for every $t \in [K_0\tau, T]$,  SDE (\ref{s1}) has a unique solution  $e^{-B\int_0^t \int_0^s \phi(s,u)du ds}\psi(t)Z_{K_0}(t)$ for every  $K_0\tau \leq t \leq T$ with Malliavin derivative satisfying (\ref{d1k}) and (\ref{d1kd}).  In this sense, it is sufficient to define 
\begin{equation}\label{defx}
 x(t) = \sum_{k=0}^{K_0-1} X_k(t) 1_{k\tau \leq t \leq (k+1)\tau} + X_{K_0}(t) 1_{K_0 \tau \leq t \leq T}. 
\end{equation}
Therefore, for every $0 \leq r \leq t-\tau$, $D^{a,b}x(t)$ should satisfy 
 \begin{align}\label{der}
 D_r^{a,b}x(t)&=D_r^{a,b}\xi_0(0)+Bx(r)+\sigma(x (r-\tau))1_{k\tau \leq r \leq t-\tau}\nonumber \\
&+\int_r^t\left\{A D_r^{a,b}x(s)+f'(x(s-\tau))D_r^{a,b}x(s-\tau)1_{0 \leq r\leq s-\tau}\right\}ds \nonumber\\
&+\int_r^t\left(B D_r^{a,b}x(s)+\sigma'(x(s-\tau))D_r^{a,b}x(s-\tau)1_{ 0 \leq r\leq s-\tau}\right)\delta^{a,b}B^{a,b}(s),~~~t\in[0,T],
 \end{align}
and for every $t-\tau <  r \leq t$,
  \begin{equation}\label{der2}
 D_r^{a,b}x(t)=Bx(r)+\int_r^t A D_r^{a,b}x(s)ds +\int_r^t B D_r^{a,b}x(s)\delta^{a,b}B^{a,b}(s),~~~t\in[0,T],
 \end{equation}
{\bf Step 4}. We continue the steps 1, 2 and 3 for Equation (\ref{der}) instead of Equation (\ref{s1}) to derive Malliavin differentiability of $D_r^{a,b}x(t)$. Finally we repeat this procedure up to the order $2M$, the order of differentiability of the functions $f$ and $g$, to deduce the assertion.
\end{proof}
We end this section by giving a recursively expression for higher order Malliavin derivatives of the solution $x(.)$,  provided we continue differentiating of Equations (\ref{d1k}) or (\ref{der}). Our expression deal with the case $0 \leq r_1, \ldots,  r_l\leq t-\tau$, the other cases have the same computation. We note that our modification will be useful to obtain some bounds for their moments in the next section.\\ 
First, let us define the processes $H_{r_l\cdots r_1},G_{r_l\cdots r_1}(s), F_{r_l\cdots r_1}(s)$ for every $2 \leq  l \leq 2M$ as follows for simplicity. 
\begin{align*}
{1}_{(1,\cdots, l)}& := 1_{r_l > (k-1)\tau} \prod_{i=1}^{l-1}1_{r_i +\tau < r_l}, \quad  {1}_{(2,\cdots, l)}:= 1_{r_1> (k-1)\tau} \prod_{i=2}^{l}1_{r_i +\tau < r_1}  \\
& {1}_{(1,\cdots, j)}= 1_{r_j > (k-1)\tau} \prod_{i=1, i \neq j}^{l-1}1_{r_i +\tau < r_j} \quad  2 \leq j \leq l-1
\end{align*}
\begin{align*}
H_{r_l\cdots r_1}(k)& := B D^{a,b}_{r_l}\cdots D^{a,b}_{r_2}X_k(r_1)+ D^{a,b}_{r_l}\cdots D^{a,b}_{r_2}\Big(\sigma(X_{k-1}(r_1-\tau))\Big)1_{(2,\cdots, l)}
+B D^{a,b}_{r_{l-1}}\cdots D^{a,b}_{r_1}X_k(r_l) \\
&+D^{a,b}_{r_{l-1}}\cdots D^{a,b}_{r_1}\Big( \sigma(X_{k-1}(r_l-\tau))\Big){1}_{(1,\cdots, l)}
+B \sum_{j=2}^{l-1}  D^{a,b}_{r_{l}}\cdots D^{a,b}_{r_{j+1}}D^{a,b}_{r_{j-1}}\cdots D^{a,b}_{r_1}X_k(r_j)  \\
&+ \sum_{j=2}^{l-1}  D^{a,b}_{r_{l}}\cdots D^{a,b}_{r_{j+1}}D^{a,b}_{r_{j-1}}\cdots D^{a,b}_{r_1}\left( \sigma(X_{k-1}(r_j-\tau))\right)1_{(1,\cdots, j)},
\end{align*}
\begin{equation*}
F_{r_l\cdots r_1}(X_{k-1},s):= D^{a,b}_{r_l}\cdots D^{a,b}_{r_1}\{f(X_{k-1}(s-\tau))\}, \qquad  G_{r_l\cdots r_1}(X_{k-1},s):= D^{a,b}_{r_l}\cdots D^{a,b}_{r_1}\{\sigma(X_{k-1}(s-\tau))\}. 
\end{equation*}
 Then $l$-th derivative for every $k\tau \leq t \leq (k+1)\tau$ as $k=1,\cdots, K_0$ and also every $0 \leq r_1, \ldots,  r_l\leq t-\tau$  satisfies
\begin{align}
D^{a,b}_{r_l}\cdots D^{a,b}_{r_1}X_k(t) & =D^{a,b}_{r_l}\cdots D^{a,b}_{r_1}X_{k-1}(k\tau)+H_{r_l\cdots r_1}(k) \nonumber \\
&+\int_{k\tau \vee r_1 \vee \cdots \vee r_l}^{t} \left\{ AD^{a,b}_{r_l}\cdots D^{a,b}_{r_1}X_k(s) +F_{r_l\cdots r_1}(X_{k-1},s) \prod_{i=1}^l 1_{r_i + \tau < s} \right\}ds  \nonumber\\
&+ \int_{k\tau  \vee r_1 \vee \cdots \vee r_l}^{t} \left\{BD^{a,b}_{r_l}\cdots D^{a,b}_{r_1}X_k(s) +G_{r_l\cdots r_1}(X_{k-1},s)  \prod_{i=1}^l 1_{r_i + \tau < s} \right\} \delta^{a,b}B^{a,b}(s),\label{dddl}
\end{align}
and for every $K_0\tau \leq t \leq T$
\begin{align*}
D^{a,b}_{r_l}\cdots D^{a,b}_{r_1}X_k(t) & =D^{a,b}_{r_l}\cdots D^{a,b}_{r_1}X_{k-1}(k\tau)+ H_{r_l\cdots r_1}(k) \nonumber \\
&+\int_{K_0\tau  \vee r_1 \vee \cdots \vee r_l}^{t} \left\{ AD^{a,b}_{r_l}\cdots D^{a,b}_{r_1}X_k(s) +F_{r_l\cdots r_1}(X_{k-1},s)\prod_{i=1}^l 1_{r_i + \tau < s}  \right\}ds  \nonumber\\
&+ \int_{K_0\tau \vee r_1 \vee \cdots \vee r_l}^{t} \left\{BD^{a,b}_{r_l}\cdots D^{a,b}_{r_1}X_k(s) +G_{r_l\cdots r_1}(X_{k-1},s)\prod_{i=1}^l 1_{r_i + \tau < s} \right\} \delta^{a,b}B^{a,b}(s).
\end{align*}
According to Step 3 in the proof of Theorem \ref{exists}, for every $0 \leq r_1, \ldots, r_l\leq t-\tau$ we have 
\begin{equation}\label{defdddx}
 D^{a,b}_{r_l}\cdots D^{a,b}_{r_1}x(t) = \sum_{k=0}^{K_0-1} D^{a,b}_{r_l}\cdots D^{a,b}_{r_1}X_k(t) 1_{k\tau \leq t \leq (k+1)\tau} + D^{a,b}_{r_l}\cdots D^{a,b}_{r_1}X_{K_0}(t) 1_{K_0\tau \leq t \leq T},
\end{equation} 
where $X_k(t)=e^{-B\int_0^t \int_0^s \phi(s,u)du ds}\psi(t)Z_{k,r_1,...,r_l}(t)$ in which 
 $$Z_{k,r_1,...,r_l}(t)=:=D^{a,b}_{r_l}\cdots D^{a,b}_{r_1}X_{k-1}(k\tau)+ H_{r_l\cdots r_1}(k)+ \int_{k\tau \vee r_1 \vee \cdots \vee r_l}^t \psi^{-1}(s) dM_{k-1,r_1,...,r_l}(s)$$ and 
\begin{align*}
M_{k,r_1,...,r_l}(t)=\int_{k\tau \vee r_1 \vee \cdots \vee r_l}^t & F_{r_l\cdots r_1}(X_k,s) ds-B\int_{k\tau \vee r_1 \vee \cdots \vee r_l}^t \int_0^s G_{r_l\cdots r_1}(X_k,s)\phi(u,s) du ds \\
&+\int_{k\tau \vee r_1 \vee \cdots \vee r_l}^t G_{r_l\cdots r_1}(X_k,s)\delta^{a,b}B^{a,b}(t).
 \end{align*}
\section{Bounds of moments to solution and its derivatives}
Construction of the solution of Equation (\ref{s1}) in the proof of Theorem \ref{exists} allow us to show that this solution and its Malliavin derivatives have uniformly bounded $p$-moments for every $p \geq 1$.
\begin{theorem}
For every $p \geq 2$, under Assumptions \ref{l548} and \ref{assum2}, there exists some positive constant $C_{0,p}$ such that 
\begin{equation}\label{ss}
\mathbb{E}\Big(\sup_{0 \leq t \leq T}\vert x(t) \vert^p\Big)  \leq C_{0,p}
\end{equation}
\end{theorem}
\begin{proof}
Proceeding the steps in the proof of Theorem \ref{exists}, induction on $k=0, \cdots, K_0$ and also the definition of $x(t)$ in Equation (\ref{defx}) show that it is sufficient to derive the uniformly boundedness of moments of processes $Z_k(.)$. To do this, for every $p \geq 2$,
from Assumption \ref{l548} and the fact $D_r\xi_0(s)=0$ for all $r \geq 0$ we obtain
\begin{align}
N_0(\tau) &:=\mathbb{E}\left(\int_{0}^{\tau}\left(\int_{0}^{\tau}\left|D_{s}^{a,b}\Big( \psi^{-1}(r)\sigma(\xi_0(r-\tau))\Big)\right|^{2/(a+b+1)}dr\right)^{\frac{p(a+b+1)}{p(a+b+1)-2}}ds\right)^{\frac{p(a+b+1)-2}{2}}\nonumber\\
&\leq \tau^2 \mathbb{E}\left(\int_{0}^{\tau}\int_{0}^{\tau}\left| D_{s}^{a,b}\psi^{-1}(s)\sigma(\xi_0(s-\tau))\right|^p drds \right)\nonumber\\
& \leq L^{\frac{p}{2}}\tau^2\mathbb{E}\left(\int_{0}^{\tau} \int_{0}^{\tau}\left|D_{s}^{a,b}\psi^{-1}(r)\right|^p(1+\xi_0(r-\tau))^{\frac{p}{2}} drds \right) \nonumber\\
& \leq  L^{\frac{p}{2}}\tau^2\mathbb{E}\left(\sup_{0 \leq r \leq \tau} (1+\xi_0(r-\tau))^{\frac{p}{2}}\int_{0}^{\tau} \int_{0}^{\tau}\left|D_{s}^{a,b}\psi^{-1}(r)\right|^p drds \right) \nonumber\\
&\leq 2^{p-2} L^{\frac{p}{2}}\tau^2\left(1+\mathbb{E}(\sup_{0 \leq r \leq \tau}\vert\xi_0(r-\tau)\vert^{p})\right)+\frac12  L^{\frac{p}{2}}\tau^4 \mathbb{E}\left( \int_{0}^{\tau} \int_{0}^{\tau}\left|D_{s}^{a,b}\psi^{-1}(r)\right|^{2p} drds \right), \label{n0}
\end{align}
Substituting $dM_0$ into the definition of $Z_0$ and then applying Corrollary \ref{corr1} and Jensen's inequality result
\begin{align}
\mathbb{E}\Big(\sup_{0 \leq t \leq \tau}\vert Z_0(t) \vert^p\Big) & \leq 2^p \Big\{\mathbb{E}(\vert \xi_0(0) \vert^p) + \mathbb{E}\Big(\sup_{0 \leq t \leq \tau}\vert \int_0^\tau \psi^{-1}(s)f(\xi_0(s-\tau))ds \vert^p \Big)\nonumber\\
&+ \mathbb{E}\Big(\sup_{0 \leq t \leq \tau}\vert \int_0^\tau \psi^{-1}(s)\sigma(\xi_0(s-\tau)) \delta^{a,b}B^{a,b}(s)\vert^p \Big)\Big\} \nonumber\\
& \leq 2^p \Big\{\mathbb{E}(\vert \xi_0(0) \vert^p) + \mathbb{E}\Big(\sup_{0 \leq t \leq \tau}\vert \psi^{-1}(s) \vert^p\vert \int_0^\tau f(\xi_0(s-\tau))ds \vert^p \Big)\nonumber\\
&+ C_1 \tau\bigg\{\left(\int_{0}^{\tau}\left|\mathbb{E}\Big( \psi^{-1}(s)\sigma(\xi_0(s-\tau))\Big)\right|^{\frac{2p}{p(a+b+1)-2}} ds\right)^{\frac{p(a+b+1)-2}{2}}+N_0(\tau)\bigg\}\Big\} \nonumber\\
& \leq  2^p \Big\{\mathbb{E}(\vert \xi_0(0) \vert^p)+ 2^{p-1} L^{p}\tau^{2p}\int_0^\tau \Big(1+\mathbb{E}(\vert\xi_0(s-\tau)\vert^{2p})\Big) ds +\frac12\mathbb{E}\left(\sup_{0 \leq s \leq \tau} \vert\psi^{-1}(s)\vert^{2p}\right) \nonumber\\
&+ C_1 \tau^2 \int_{0}^{\tau}\Big\vert\mathbb{E}\Big( \psi^{-1}(s)\sigma(\xi_0(s-\tau))\Big)\Big\vert^p ds +C_1\tau N_0(\tau)\Big\}\nonumber\\
&\leq 2^p \Big\{\mathbb{E}(\vert \xi_0(0) \vert^p)+ 2^{p-1} L^{p}\tau^{2p}\int_0^\tau \Big(1+\mathbb{E}(\vert \xi_0(s-\tau \vert^{p})\Big) ds +\frac12\mathbb{E}\left(\sup_{0 \leq s \leq \tau}\vert \psi^{-1}(s)\vert^{2p}\right) \nonumber\\
&+\frac12 C_1 \tau^2 \int_{0}^{\tau}\mathbb{E} (\vert\psi^{-1}(s)\vert^{2p}) ds+\frac12 C_1\tau^2 L^p 2^p \int_{0}^{\tau}\mathbb{E}\Big(  1+\vert \xi_0(s-\tau)\vert^{2p}\Big) ds +C_1\tau N_0(\tau)\Big\}, \label{z0}
\end{align}
where we used the fact $(z+y)^p \leq 2^p( z^p+y^p)$ several times in the above inequalities. Finally, substitue (\ref{n0}) in (\ref{z0}) and then use Assumption \ref{assum2} and Proposition \ref{boundpsi} to deduce that for every $0 \leq t \leq \tau$ the solution $x(t)=e^{-B\int_0^t \int_0^s \phi(s,u)du ds}\psi(t) Z_0(t)$ has uniformly bounded moments.\\
On replacing $\xi_0$ by  $X_k$ and repeating a similar computation recurcively on $k$, one derive that for every $k=1, \cdots, K_0-1$
\begin{align*}
N_k(k\tau) &:=\mathbb{E}\left(\int_{k\tau}^{(k+1)\tau}\left(\int_{k\tau}^{(k+1)\tau} D_s^{a,b}\Big(\psi^{-1}(r)\sigma(X_{k-1}(r-\tau))\Big)^{2/(a+b+1)}dr\right)^{\frac{p(a+b+1)}{p(a+b+1)-2}}ds\right)^{\frac{p(a+b+1)-2}{2}}\nonumber\\
&\leq \tau^2 \mathbb{E}\left(\int_{k\tau}^{(k+1)\tau}\int_{k\tau}^{(k+1)\tau}\left| \psi^{-1}(r)\sigma'(X_{k-1}(r-\tau))D_{s}^{a,b}X_{k-1}(r-\tau)+ D_{s}^{a,b}\psi^{-1}(r)\sigma(X_{k-1}(s-\tau))\right|^p dr ds\right)\nonumber\\
& \leq 2^{p-1}\tau^2 K_2^p\mathbb{E}\left(\sup_{k\tau \leq s \leq (k+1)\tau} \vert\psi^{-1}(s)\vert^{2p}\right) +2^{p-1}\tau^2 K_2^p\mathbb{E}\left(\int_{k\tau}^{(k+1)\tau}\int_{k\tau}^{(k+1)\tau}\left| D_{s}^{a,b}X_{k-1}(r-\tau)\right|^{2p} drds \right) \nonumber\\
&+2^{2p-1}L^p\tau^2\left(1+\mathbb{E}(\sup_{k\tau \leq r \leq (k+1)\tau}\vert X_{k-1}(r-\tau)\vert^{p})\right)+2^{p-1}\tau^2 \mathbb{E}\left( \int_{k\tau}^{(k+1)\tau}\int_{k\tau}^{(k+1)\tau}\left|D_{s}^{a,b}\psi^{-1}(r)\right|^{2p} drds \right), 
\end{align*}
and also 
\begin{align*}
\mathbb{E}\Big(\sup_{k\tau \leq t \leq (k+1)\tau}\vert Z_k(t) \vert^p\Big) 
& \leq 2^p \Big\{\mathbb{E}(\vert X_{k-1}(k\tau) \vert^p)+ 2^{p-1} L^{p}\tau^{2p}\int_{k\tau}^{(k+1)\tau} \Big(1+\mathbb{E}(\vert X_{k-1}(s-\tau)\vert^{p})\Big) ds \nonumber\\
&+\frac12\mathbb{E}\left(\sup_{k\tau \leq s \leq (k+1)\tau}\vert \psi^{-1}(s)\vert^{2p}\right) 
+\frac12 C_1 \tau^2  \int_{k\tau}^{(k+1)\tau}\mathbb{E} (\vert\psi^{-1}(s)\vert^{2p}) ds \nonumber\\
&+\frac12 C_1\tau^2 L^p 2^p \int_{k\tau}^{(k+1)\tau}\mathbb{E}\Big(  1+\vert X_{k-1}(s-\tau)\vert^{2p}\Big) ds +C_1\tau N_k(k\tau)\Big\}.
\end{align*}
Hence, Proposition \ref{boundpsi} and induction on $k$ deduce the boundedness of moments of $x(.)$ in $[0, K_0\tau]$ and finally by repeating this computation for every $K_0\tau  \leq t \leq T$, the claim can be obtained. 
\end{proof}
In sequence, since the following computations and results can be exactly repeat for every $K_0 \tau \leq t \leq T$ and $
t -\tau  \leq r \leq t$, we just demonstrate the results on $0 \leq t \leq K_0\tau$ and $0 \leq r \leq t-\tau$  as follows.\\
From the definition of the functions $F_{r_l\cdots r_1}(X_k,t) $ and  $G_{r_l\cdots r_1}(X_k,t) $, we understand that these processes  depend on higher derivatives of the functions $f$ and $\sigma$, respectively, and Malliavin derivatives of $X_{k-1}$ up to the order $l$. Then
$$\vert F_{r_l\cdots r_1}(X_{k-1},t) \vert^p=\mathcal{P}_1\left(f'(X_{k-1}), \cdots, f^{(l)}(X_{k-1}), D_{r_1}^{a,b}X_{k-1}, \cdots, D^{a,b}_{r_l}\ldots D^{a,b}_{r_1} X_{k-1}\right)(t),$$
$$\vert G_{r_l\cdots r_1}(X_{k-1},t) \vert^p=\mathcal{P}_2\left(\sigma'(X_{k-1}), \cdots, \sigma^{(l)}(X_{k-1}), D_{r_1}^{a,b}X_{k-1}, \cdots, D^{a,b}_{r_l}\ldots D^{a,b}_{r_1} X_{k-1}\right)(t),$$
where $\mathcal{P}_1$ and $\mathcal{P}_2$ are polynomial functions. From Assumption \ref{l548}  we derive that for every $p \geq 2$, $0 \leq r_1, \cdots r_l \leq T$ and $0 \leq t_1 \leq t_2 \leq T$
\begin{equation}\label{boundF}
\mathbb{E}\Big(\sup_{t_1 \leq t \leq t_2} \vert  F_{r_l\cdots r_1}(X_k,t) \vert^p\Big)+\mathbb{E}\Big(\sup_{t_1 \leq t \leq t_2} \vert  G_{r_l\cdots r_1}(X_k,t) \vert^p\Big)  \leq 2^{lp}K_2\sum_{j=2}^l  \mathbb{E}\Big(\sup_{t_1 \leq t \leq t_2}\vert  D_{r_j}^{a,b}\cdots D^{a,b}_{r_1}X_{k-1}(t-\tau) \vert^{q'_j} \Big)
\end{equation} 
where integer numbers $q'_j$ satisfy $1 \leq q'_j \leq p$. 
Furthermore, when $l< 2M$, for every $0 \leq r\leq t-\tau$
\begin{equation}\label{GGG}
\mathbb{E}\Big(\sup_{t_1 \leq t \leq t_2} \vert D^{a,b}_r G_{r_l\cdots r_1}(X_k,t) \vert^p\Big)  \leq  2^{lp}K_2\sum_{j=1}^{l+1}  \mathbb{E}\Big(\sup_{t_1 \leq t \leq t_2}\vert  D_{r_j}^{a,b}\cdots D^{a,b}_{r_1}X_{k-1}(t-\tau) \vert^{q_{j,G}} \Big)
\end{equation}
in which integer numbers $q_{j,G}$ satisfy $1 \leq q_{j,G} \leq p$. Define the vectors 
\begin{align*}
V_i &:=\Big(\sigma^{(i)}(X_{k-1}(r_1-\tau)), \cdots, \sigma^{(i)}(X_{k-1}(r_l-\tau))\Big), \\
W_{j,1}& :=\Big( D_{r_j}^{a,b}\cdots D^{a,b}_{r_2}X_{k}(r_1), \cdots,  D_{r_j}^{a,b}\cdots D^{a,b}_{r_2}X_{k}(r_l)\Big),\nonumber\\
W_{j,2}&:=\Big(D_{r_j}^{a,b}\cdots D^{a,b}_{r_2}X_{k-1}(r_1-\tau), \cdots,  D_{r_j}^{a,b}\cdots D^{a,b}_{r_2}X_{k-1}(r_l-\tau)\Big)\\
w_{j,j'}&:=\Big( D_{r_{j'}}^{a,b}\cdots D^{a,b}_{r_{j+1}}D_{r_{j-1}}^{a,b}\cdots D^{a,b}_{r_1}X_{k-1}(r_1-\tau), \cdots, D_{r_{j'}}^{a,b}\cdots D^{a,b}_{r_{j+1}}D_{r_{j-1}}^{a,b}\cdots D^{a,b}_{r_1}X_{k-1}(r_l-\tau)\Big),\\
w^{(1)}_{j,j'}&:=\Big( D_{r_{j'}}^{a,b}\cdots D^{a,b}_{r_{j+1}}D_{r_{j-1}}^{a,b}\cdots D^{a,b}_{r_1}X_{k}(r_1-\tau), \cdots, D_{r_{j'}}^{a,b}\cdots D^{a,b}_{r_{j+1}}D_{r_{j-1}}^{a,b}\cdots D^{a,b}_{r_1}X_{k}(r_l-\tau)\Big),\\
v_{j''}& :=\Big( D_{r_{j''}}^{a,b}\cdots D^{a,b}_{r_1}X_{k-1}(r_1-\tau), \cdots,  D_{r_{j''}}^{a,b}\cdots D^{a,b}_{r_1}X_{k-1}(r_l-\tau)\Big),\\
v^{(1)}_{j''}& :=\Big( D_{r_{j''}}^{a,b}\cdots D^{a,b}_{r_1}X_{k}(r_1-\tau), \cdots,  D_{r_{j''}}^{a,b}\cdots D^{a,b}_{r_1}X_{k}(r_l-\tau)\Big),\nonumber
\end{align*}
for every $1 \leq i \leq l$, $2 \leq j \leq l$, $3 \leq j' \leq l$ and  $1 \leq j'' \leq l-1$. Therefore $\vert H_{r_l\cdots r_1}\vert^p$ would be a polynomial composition of the elements of these vectors; i.e., there exists some polynomial $\mathcal{P}_3$  such that 
\begin{equation*}
\vert H_{r_l\cdots r_1}\vert^p=\mathcal{P}_3\left(V_1, \cdots, V_l, W_{2,1}, W_{2,2} \cdots, W_{l,1}, W_{l,2}, w_{2,3}, w^{(1)}_{2,3}, \cdots, w_{l,l}, w^{(1)}_{l,l}, v_1, v^{(1)}_1, \cdots, v_{l-1}, v^{(1)}_{l-1}\right).
\end{equation*} 
We again conclude from Assumption \ref{l548} that for every $p \geq 2$ there exists some constant $c_{p,H_1}$ such that
\begin{align}\label{boundH}
\mathbb{E}\Big(\vert  H_{r_l\cdots r_1} \vert^p\Big) \leq c_{p,H_1} \Big(& \sum_{i=1}^{l}  \mathbb{E}(\vert V_i \vert^{q_{i}})+\sum_{j=1}^{l-1} \mathbb{E}(\vert v_{j''} \vert^{q_{j''}}+\vert v^{(1)}_{j''} \vert^{q_{j''}})+\sum_{j'=3}^{l}\sum_{j=2}^{l} \Big\{ \mathbb{E}(\vert W_{j,1} \vert^{q_{j,1}})+\mathbb{E}(\vert W_{j,2} \vert^{q_{j,2}})\Big\}  \nonumber\\
& +\sum_{j'=3}^{l}\sum_{j=2}^{l} \Big\{ \mathbb{E}(\vert w_{j,j'} \vert^{q_{j,j'}}+\vert w^{(1)}_{j,j'} \vert^{q_{j,j'}})\Big\}\Big),
\end{align}
 where all integer numbers $q_{j}, q_{j,1}, q_{j,2}, q_{j,j'}$ and $q_{j''}$ are between 1 and $p$. \\
It is worth mentioning that the order of differentiability in components of vectors $V_i, W_{j,1}, W_{j,2}, w_{j,j'},  w^{(1)}_{j,j'}$ and $v_{j'}, v^{(1)}_{j'}$ are utmost of the order $l-1$. This fact helps us to find some upper bounds for the moments of $D^{a,b}_{u_l}..D^{a,b}_{u_1}x(.)$ that will be demonstrated in Theorem \ref{upperderi}.\\
 Now, we are ready to prove some upper bounds for the moments of higher derivatives of the solution to SDDE (\ref{s1}). 
\begin{theorem}\label{upperderi}
 Under Assumptions \ref{l548} and \ref{assum2}, for every $p \geq 2$ and $1 \leq l \leq M$ there exists some positive constant $C'_{l,p}$ such that $l$-th  Malliavin derivative of the solution to SDE (\ref{s1}) satisfies the following inequlity.
\begin{equation}\label{ssd}
\mathbb{E}\Big(\sup_{0 \leq t \leq T}\sup_{0 \leq u_1, \cdots,  u_l \leq t} \vert D^{a,b}_{u_l}..D^{a,b}_{u_1} x(t)\vert^p  \Big) \leq C'_{l,p}
\end{equation}
\end{theorem}
\begin{proof}
 We show the assertion by induction on $k$. We first take the following computation for every $k=0, \cdots, K_0$ and $1 \leq l  <  2M$. Assume $\sup_{0 \leq u_1, \cdots,  u_l \leq t} \vert D^{a,b}_{u_l}..D^{a,b}_{u_1} X_{k-1}(t)\vert^p$ is achieved in $u_1=r_1, \cdots , u_l=r_1$, which we omitted dependency of $r_i$s' to $t$ for simplicity without loss of generality. From Equations (\ref{GGG})  and (\ref{boundF})
\begin{align}
2^{-p+1} DN_{k,r_1,...,r_l}& (k\tau) := 2^{-p+1}\mathbb{E}\left(\int_{k\tau}^{(k+1)\tau}\left(\int_{k\tau}^{(k+1)\tau} \vert D_r^{a,b}\Big(\psi^{-1}(s)G_{r_l\cdots r_1}(X_{k-1},s)\Big)\vert^{2/(a+b+1)}ds\right)^{\frac{p(a+b+1)}{p(a+b+1)-2}}dr\right)^{\frac{p(a+b+1)-2}{2}}\nonumber\\
& \leq \tau^4 \mathbb{E}\left(\sup_{0 \leq s \leq T} \vert \psi^{-1}(s)\vert^{2p}\right) +\tau^2 \mathbb{E}\left(\int_{k\tau}^{(k+1)\tau}\int_{k\tau}^{(k+1)\tau}\left| D_{r}^{a,b}G_{r_l\cdots r_1}(X_{k-1},s)\right|^{2p} drds \right) \nonumber\\
&+ \tau^2\mathbb{E}\left(\int_{k\tau}^{(k+1)\tau}\int_{k\tau}^{(k+1)\tau} \vert D_{r}^{a,b}\psi^{-1}(s) \vert^{2p} drds\right) +\tau^2 \mathbb{E}\left(\int_{k\tau}^{(k+1)\tau}\int_{k\tau}^{(k+1)\tau}\left| G_{r_l\cdots r_1}(X_{k-1},s)\right|^{2p} drds \right) \nonumber\\
& \leq \tau^4 \mathbb{E}\left(\sup_{0 \leq s \leq T} \vert \psi^{-1}(s)\vert^{2p} \right)+\tau^2\mathbb{E}\left(\int_{k\tau}^{(k+1)\tau}\int_{k\tau}^{(k+1)\tau}  D_{r}^{a,b}\vert \psi^{-1}(s)\vert^{2p}  drds\right)  \nonumber\\
&+2\tau^2 2^{2lp}K_2^2\sum_{j=1}^{l+1}  \mathbb{E}\Big(\sup_{k\tau \leq s \leq (k+1)\tau}\vert  D_{r_j}^{a,b}\cdots D^{a,b}_{r_1}X_{k-1} (s-\tau)\vert^{2q_{j,G}}\Big). \label{ddnn}
\end{align}
On the other hand, from Theorem \ref{l0} 
\begin{align}
\mathbb{E}\Big(\sup_{k\tau \leq t \leq (k+1)\tau}  \vert Z_{k,r_1,...,r_l}(t) \vert^p\Big)
& \leq 3^p \bigg\{2^p\mathbb{E}(\vert D^{a,b}_{r_l}\cdots D^{a,b}_{r_1}X_{k-1}(k\tau)\vert^p+ 2^p\mathbb{E}(\vert  H_{r_l\cdots r_1}(k)\vert^p)\nonumber\\
&+\frac12 \mathbb{E}\left(\int_{k\tau}^{(k+1)\tau}  \vert\psi^{-1}(s)\vert^{2p}  \right)+\frac12 \mathbb{E}\left(\int_{k\tau}^{(k+1)\tau}  \vert  F_{r_l\cdots r_1}(X_{k-1},s) \vert^{2p} ds \right) \nonumber\\
&+ \mathbb{E}\left(\sup_{k\tau \leq t \leq (k+1)\tau} \vert \int_{k\tau}^{t}   \psi^{-1}(s)G_{r_l\cdots r_1}(X_{k-1},s)  \delta^{a,b}B^{a,b}(s) \vert^{p} \right)\bigg\}\nonumber\\
& \leq 3^p \bigg\{2^p\mathbb{E}(\vert D^{a,b}_{r_l}\cdots D^{a,b}_{r_1}X_{k-1}(k\tau)\vert^p+ 2^p\mathbb{E}(\vert  H_{r_l\cdots r_1}(k)\vert^p)\nonumber\\
&+\frac12 \tau \mathbb{E}\Big(\sup_{k\tau \leq s \leq (k+1)\tau} \vert  F_{r_l\cdots r_1}(X_{k-1},s) \vert^{2p}  \Big) 
+ \frac12(1+ C_1 \tau) \mathbb{E}\left(\int_{k\tau}^{(k+1)\tau}  \vert\psi^{-1}(s)\vert^{2p} ds \right) \nonumber\\
& + \frac12 C_1 \tau \int_{k\tau}^{(k+1)\tau}  \mathbb{E} (\vert\ G_{r_l\cdots r_1}(X_{k-1},s)  \vert^{2p} ) + C_1 \tau DN_{k,r_1,...,r_l}(k\tau)  \bigg\}.\label{zzzp}
\end{align}
Now, sustitute Equations (\ref{boundH}), (\ref{boundF}) and (\ref{ddnn}) in (\ref{zzzp})  and result
\begin{align}
\mathbb{E}\Big(\sup_{k\tau \leq t \leq (k+1)\tau}& \vert Z_{k,r_1,...,r_l}(t) \vert^p\Big) \nonumber \\
&\leq  3^p \bigg\{2^p\mathbb{E}(\vert D^{a,b}_{r_l}\cdots D^{a,b}_{r_1}X_{k-1}(k\tau)\vert^p )+2^p c_{p,H_1} \sum_{j'=3}^{l}\sum_{j=2}^{l} \Big\{ \mathbb{E}(\vert w_{j,j'} \vert^{q_{j,j'}}+\vert w^{(1)}_{j,j'} \vert^{q_{j,j'}})\Big\}\nonumber\\
&+2^p c_{p,H_1} \left(\sum_{i=1}^{l}  \mathbb{E}(\vert V_i \vert^{q_{i}})+\sum_{j=1}^{l-1} \mathbb{E}(\vert v_{j''} \vert^{q_{j''}}+\vert v^{(1)}_{j''} \vert^{q_{j''}})+\sum_{j'=3}^{l}\sum_{j=2}^{l} \Big\{ \mathbb{E}(\vert W_{j,1} \vert^{q_{j,1}})+\mathbb{E}(\vert W_{j,2} \vert^{q_{j,2}})\Big\}\right)\nonumber\\
&+\frac12\tau (1+C_1\tau) \mathbb{E}\left(\sup_{0 \leq s \leq T}\vert \psi^{-1}(s)\vert^{2p}\right)\nonumber\\
&+2^{2lp}\frac12 (1+C_1\tau)\tau K_2^2\sum_{j=2}^{l}  \mathbb{E}\Big(\sup_{k\tau \leq t \leq (k+1)\tau}\vert  D_{r_j}^{a,b}\cdots D^{a,b}_{r_2}X_{k-1}(t) \vert^{q_j} \Big) \nonumber\\
&+2^{p}C_1\tau^5\mathbb{E}\left(\sup_{0 \leq s \leq T}\vert \psi^{-1}(s)\vert^{2p}\right) +2^p C_1\tau^3 \mathbb{E}\left(\int_{k\tau}^{(k+1)\tau}\int_{k\tau}^{(k+1)\tau} \vert D_{r}^{a,b}\psi^{-1}(s)\vert^{2p} drds\right)  \nonumber\\
&+2^{p+1}C_1\tau^3 2^{2lp}K_2^2\sum_{j=2}^{l+1}  \mathbb{E}\Big(\sup_{k\tau \leq s \leq (k+1)\tau}\vert  D_{r_j}^{a,b}\cdots D^{a,b}_{r_1}X_{k-1} (s-\tau)\vert^{2q_{j,G}}\Big)\label{z1l}
\end{align}
Here it is crucial to mention that in this computation the order of deriatives for all terms in the right hand side of Equation (\ref{z1l}) are utmost of the order $l$, except of the last term which involves $ D_{r_{l+1}}^{a,b}\cdots D^{a,b}_{r_1}X_{k-1} (s-\tau)$. So, we first assume that $k=0$. Applying Assumption \ref{assum2}, Proposition \ref{boundpsi} in Equation (\ref{z1l}) and induction on $l$ deduce that the process $X_1$ is $2M$-times differentiable and their derivatives have bounded $p$-moments. Following by induction on $k$ to result that the process $X_{k+1}$ has Malliavin derivatives of one order less than those of $X_k$ and thier derivatives have uniformly bounded moments. Since $K_0+1 < 2M$, it is sufficient to proceed previous stage $M$ times to obtain the assertion.   
\end{proof}
%%%%%%%%%%%%%%%%%%%%%%%%%%%%
\section{Regularity of the density}
To achieve the reqularity of the density of solution to SDDE (\ref{s1}), we assume that the functions $f$ and $\sigma$ are infinitely differentiable with bounded derivatives; in fact $M=\infty$ in Assumption \ref{l548}. Also we consider the following hypothesis.\\
{\bf Hopothesis H}: There exists some constant $M_0$ such that $\vert \sigma(x) \vert > M_0$ for all $x$.
\begin{theorem}
Under Hypothsis H and assumption $M=\infty$, for every $t \in [0,T]$ the solution of the SDDE (\ref{s1}) has an infinitely differentiable density with respect to Lebesgue's measure on $\mathbb{R}$.
\end{theorem}
\begin{proof}
Fixed $t \in [0,T]$. Theorem \ref{exists} guarantees the infinitly Malliavin diffrentiability of the solution $x(t)$ as $M=\infty$. To prove the second part, we have to show that $\mathbb{E}\Big(\int_0^T \vert D_uX(t)\vert^2 du\Big) < \infty$.  By using Malliavin's criterion it is sufficient to check that for every $p \geq 1$ there exists some $\epsilon_0 > 0$ such that for all $\epsilon_1 \leq \epsilon_0$; 
\begin{equation*}
P\Big(\int_0^T \vert D_u X(t) \vert^2 du < \epsilon_1 \Big) < \epsilon_1^p
\end{equation*}
To this end, we know 
\begin{align*}
P\Big(\int_0^T & \vert D_u X(t) \vert^2 du < \epsilon_1 \Big) \\
& \leq P\Big(\int_{t-\epsilon_1^\theta}^{t} \vert \psi(t) \vert^2\vert \psi(u) \vert^{-2} \vert Bx(u) + \sigma(x(u-\tau)1_{r > \tau}\vert^2  du < \epsilon_1 \Big) \\
&\leq P(A_0, \inf_{0 \leq t \leq T} \psi(r) > \delta, \sup_{0 \leq t \leq T} \psi(r) <  \delta^{-1} ) 
+  P(\inf_{0 \leq t \leq T} \psi(r) \leq  \delta)+P(\sup_{0 \leq t \leq T} \psi(r) \geq \delta^{-1})\\
&:=P_{1,\epsilon_1}+P_{2,\epsilon_1}+P_{3,\epsilon_1},
\end{align*}
with $A_0=\left\{ \int_{t-\epsilon_1^\theta}^{t} \vert \psi(t) \vert^2\vert \psi(u) \vert^{-2} \vert Bx(u) + \sigma(x(u-\tau)1_{r > \tau}\vert^2  du < \epsilon_1 \right\}$.
Let $\delta=\epsilon_1^{\frac18}$ and apply Chebyshev's inequality to show that $P_{3,\epsilon_1} \leq \epsilon_1^p$. Indeed, 
\begin{equation*}
P_{3,\epsilon_1} \leq \delta^{8p}\mathbb{E}\Big( \sup_{0 \leq t \leq T} \vert \psi(t) \vert^{8p} \Big) \leq \delta^{8p}C_{8p}.
\end{equation*}
From Proposition \ref{boundpsi}, the  infimum part being a simple modification of this argument for $\psi^{-1}(.)$ and therefore $P_{2,\epsilon_1} \leq \epsilon_1^p$.
 On the other hand, 
\begin{align*}
P_{1,\epsilon_1}& \leq  P\Big(\int_{t-\epsilon_1^\theta}^{t}\vert Bx(u) + \sigma(x(u-\tau)1_{r > \tau}\vert^2  du < \epsilon_1\delta^{-4} \Big) \\
& \leq  P\Big(\int_{t-\epsilon_1^\theta}^{t}\vert Bx(u) + \sigma(x(u-\tau)1_{r > \tau})\vert^2  du < \epsilon_1\delta^{-4}, \int_{t-\epsilon_1^\theta}^{t} \vert Bx(u) \vert \leq \epsilon_1^\gamma \Big) \\
&+  P\Big(\int_{t-\epsilon_1^\theta}^{t} \vert Bx(u) \vert > \epsilon_1^\gamma \Big).
\end{align*}
Since $\vert \sigma(y) \vert > 0$ for all $y$, when $\theta < \frac12$ one can easily see that the first summand in above inequality is vanished. By using Chebyshev's inequality and applying Equation (\ref{ss}), for every $q>1$ 
\begin{align*}
  P\Big(\int_{t-\epsilon_1^\theta}^{t} \vert Bx(u) \vert > \epsilon_1^\gamma \Big) \leq \frac{1}{\epsilon_1^{\gamma q}}\mathbb{E}\Big( \sup_{t-\epsilon_1^\theta \leq u \leq t} \vert Bx(u) \vert^q \Big) \leq B^qC_{0,q}\epsilon_1^{(\theta-\gamma) q}.
\end{align*}
So, choosing $0 < \gamma < \theta < \frac12$ result the assertion.
\end{proof}

\section{Conclusion}
In this article, we demonstrate the problem of existence and uniqueness of solution to a stochastic differential equations with delay driven by weighted fractional Brownian motion. We introduce the solution step by step when proceeding with delay in time. This solution is Malliavin diffrentiable of higher order and its Malliavin derivatives have uniformly bounded moments. The solution has an infinitely differentiable density with respect to Lebesgue's measure on $\mathbb{R}^d$ for $d\geq 1$. Our result allow one apply this result in many applications such as numerical methods, for instance one can easily  check that Euler approximation process of the SDDE (\ref{s1}),  in continuous version, has Malliavin derivatives with uniformly bounded moments. Also, one can result that the law of Euler approximation process is smooth, which is essential to show weak convergence of approximated process to the true solution. \\

{\bf Acknowledgement}: 
This research did not receive any specific grant from funding agencies in the public, commercial, ornot-for-profit sectors.\\


\begin{thebibliography}{00}
\bibitem{34}{A.M. Abylayeva, R. Oinarov, and L.-E. Persson. Boundedness and compactness of a class of hardy type operators. Lule{\aa} University of Technology, Graphic Production (2016).}
 \bibitem{29}{E. Al\`{o}s, O. Mazet, and  D. Nualart, Stochastic calculus with respect to Gaussian processes. Ann. Probab,  29 (2001) pp. 766-801.}
 \bibitem{33}{E. Al\`{o}s, and D. Nualart, Stochastic integration with respect to the fractional Brownian motion. Stoch Stoch
Reports, 75 (2003) pp. 129-152.}
\bibitem{Elwo}{O. Amin, E. Coffie, F. Harang, and F. Proske, A Bismut-Elworthy-Li formula for singular SDE's driven by a fractional Brownian motion and aplicationsto rough volatility modeling, arxiv:1805.11435v1, (2018).}
\bibitem{BaH07}{F. Baudoin, and M. Hairer, A version of H\"{o}rmander's theorem for the fractional Brownian motion. Probab.
Theory Relat. Fields, 139 (2007) pp. 373-395.}
\bibitem{28}{R. Belfadi, K. Es-Sebaiy, and Y. Ouknine, Parameter estimation for fractional Ornstein-Uhlenbeck process: Nonergodic
Case, Front Sci. Eng., 1 (2011) pp. 1-16.}
\bibitem{3}{ F. Biagini, Y. Hu, B. Øksendal, and T. Zhang, Stochastic Calculus for fBm and Applications, Probability and Its Application, Springer, Berlin, (2008).}
\bibitem{BoGo08}{T. Bojdecki, LG. Gorostiza and A. Talarczyk, Occupation time limits of inhomogeneous Poisson systems of independent particles. Stoch. Process. Appl. 118 (2008) pp. 28-52.}
\bibitem{BoGo07}{T. Bojdecki, LG. Gorostiza, and A.Talarczyk, Some extensions of fractional Brownian motion and sub-fractional Brownian motion related to particle system. Electron. Commun. Probab. 12 (2007) pp. 161-172.}
\bibitem{BoGo082}{T. Bojdecki, LG. Gorostiza, and A.Talarczyk, Self-similar stable processes arising from high density limits of occupation times of particle systems. Potential Anal. 28 (2008) pp. 71-103.}
\bibitem{BouCar15}{A. Boudaoui, T. Caraballo, and A. Ouahab, Existence of mild solutions to stochastic delay
evolution equations with a fractional Brownian
motion and impulses, Stochastic Analysis and Applications, 33 (2015) pp. 244-258.}
\bibitem{Bou12}{B. Boufoussi, S. Hajji and El H. Lakhel, Functional differential equations in Hilbert spaces driven
by a fractional Brownian motion, Afr. Mat. 23 (2012) pp. 173-194.}
\bibitem{Cara11}{T. Caraballo, M.J. Garrido-Atienza, and T. Taniguchi, The existence and exponential behavior of solutions to stochastic delay evolution equations with a fractional Brownian motion, Nonlinear Analysis, 74 (2011) pp. 3671-3684.}
\bibitem{Duncan}{E. Duncan, Y. Hu, and B. Pasik-Duncan, Stochastic calculus for fractional Brownian motion I. Theory, SIAM J. CONTROL OPTIM., 38 no.2 (2000) pp. 582-612.}
\bibitem{FR06}{M. Ferrante, and C. Rovira, Stochastic delay differential equations driven by fractional Brownian motion
with Hurst paramete $H >1/2$. Bernoulli, 12  no.1 (2006) pp. 85-100.}
\bibitem{Ga09}{J. Garzón, Convergence to weighted fractional Brownian sheets. Commun. Stoch. Anal. 3 (2009)1-14.}
\bibitem{5}{ Y. Hu, Integral transformations and anticipative calculus for fractional Brownian motions, Mem. Amer. Math. Soc., 175  no. 825 (2005) pp. viii-127.}
\bibitem{flinear}{Y. Hu, and X-Y. Zhou, Stochastic control for linear systems driven by fractional noises, SIAM Journal on Control and Optimization, 43,  no. 6, (2005) pp. 2245-2277.}
\bibitem{HN06}{Y. Hu, and D. Nualart, Differential equations driven by h\"{o}lder continuous functions
of order greater than 1/2, In Stochastic Analysis and Applications. Abel Symp., 2  (2007) pp. 399-413, Springer, Berlin.}
\bibitem{Tin12}{J. A. León,  and S. Tindel, Malliavin Calculus for Fractional Delay Equations, J. Theor Probab, 25 (2012) pp. 854-889.}
\bibitem{Lyo94} {T. J. Lyons, Differential equations driven by rough signals. I. An extension of an
inequality of L. C. Young. Math. Res. Lett., 1  no. 4 (1994) pp. 451-464.}
\bibitem{NeuNor08}{A. Neuenkirch, I. Nourdin, and S. Tindel, Delay equations driven by rough paths. Electron. J. Probab.,
13  no.67 (2008) pp. 2031-2068.}
\bibitem{4}{Y. S. Mishura, Stochastic Calculus for Fractional Brownian Motion and Related Processes,
Lect. Notes Math., Vol. 1929, Berlin, Heidelberg, Springer, (2008).}
\bibitem{Nor95}{I. Norros, On the use of fractional Brownian motion in the theory of connectionless networks. IEEE
J. Selected Areas Commun., 13 no.6 (1995) pp. 953-962.}
 \bibitem{NS06}{ I. Nourdin, and T. Simon, On the absolute continuity of one-dimensional SDEs
driven by a fractional Brownian motion, Statistics and Probability Letters, 76 (2006) pp. 907-912.}
\bibitem{2}{D. Nualart, Malliavin Calculus and Related Topics, 2nd ed., Springer, New York, (2006).}
\bibitem{Nua1998}{ D. Nualart, Analysis on Wiener space and anticipating stochastic calculus. Lecture Notes
in Math., 1690 (1998) pp. 123-227.}
\bibitem{Nua2002} {D. Nualart,  and Y. Ouknine, Regularization of differential equations
by fractional noise. Stochastic Process. Appl., 102 (2002) pp. 103-116.}
\bibitem{Nua2003} {D. Nualart, and Y. Ouknine, Stochastic differential equations with
additive fractional noise and locally unbounded drift, Progress in Probability, 56 (2003) pp. 353- 365.}
\bibitem{NuaRas2002} {D. Nualart, and A. Rascanu, Differential equations driven by fractional
Brownian motion. Collect. Math., 53 (2002) pp. 55-81.}
\bibitem{NS05}{D. Nualart, and B. Saussereau, Malliavin calculus for stochastic differential equations
driven by a fractional Brownian motion, Stoch. Process. Appl., 119 no.2 (2009) pp. 391-409.}
\bibitem{32}{V. Pipiras, and MS. Taqqu,  Integration questions related to fractional Brownian motion. Probab. Theory Relat. Fields, 118 (2000) pp. 251-291.}
\bibitem{10}{O. Sheluhin, S. Smolskiy, and A. Osin, Self-Similar Processes in Telecommunications, John
Wiley Sons, Inc, New York, (2007).}
\bibitem{Shen20}{G-J Shen, L-T Yan, and J. Cui, Berry-Esséen bounds and almost sure CLT for quadratic variation of weighted fractional Brownian motion, Journal of Inequalities and Applications, (2013) pp. 2013:275.}
\bibitem{Shen16}{G-J. Shen, X-W. Yin, and L-T Yan, Least squares estimation for Ornstein-Uhlenbeck processes driven by the weighted fractional Brownian motion. Acta Math. Sci. 36 (2016) pp. 394-408.}
\bibitem{Shi99}{A. N. Shiryaev, Essentials of stochastic finance. In: Advanced Series on Statistical Science and Applied
Probability, vol. 3. World Scientific Publishing Co. Inc., River Edge. Facts, models, theory. Translated
from the Russian manuscript by N. Kruzhilin (1999).}
\bibitem{Tin09}{S. Tindel, and I. Torrecilla,  Some differential systems driven by fBm with Hurst parameter greater than 1/4, In: Stochastic analysis and related topics, L. Decreusefond and J.Najim (2012), chapter 8, pp.169-202.}
\bibitem{Wei07}{F. Wei, and K. Wang, The existence and uniqueness of the solution
for stochastic functional differential equations with infinite delay,
J. Math. Anal. Appl., 331 (2007) pp. 516-531.}
\bibitem{YaWa14}{L-T. Yan, Z. Wang, and H. Jing, Some path properties of weighted fractional Brownian motion. Stochastics 86 (2014) pp. 721-758.}
\bibitem{Zah01}{ M. Z\"{a}hle, Integration with respect to fractal functions and stochastic calculus. II.
Mathematische Nachrichten 225 (2001) pp. 145-183.}
\bibitem{zh20}{S.Q. Zhang, and C.I. Yuan, Stochastic differential equations driven by
fractional Brownian motion with locally Lipschitz
drift and their implicit Euler approximation, Proceedings of the Royal Society of Edinburgh, (2020) pp. 1-27.}
 \end{thebibliography}
 \end{document}